\documentclass[12pt,a4paper,leqno]{article}

\usepackage[nottoc]{tocbibind}
\usepackage{latexsym}
\usepackage{amssymb,exscale}
\usepackage[centertags]{amsmath}
\usepackage{amsthm}
\usepackage{paralist}
\usepackage{footnpag}
\usepackage[all]{xy}

\usepackage[linktocpage]{hyperref}

\numberwithin{equation}{section}

\swapnumbers
\theoremstyle{definition}
\newtheorem{theorem}[equation]{Theorem}
\newtheorem{lemma}[equation]{Lemma}
\newtheorem{example}[equation]{Example}
\newtheorem{definition}[equation]{Definition}
\newtheorem{proposition}[equation]{Proposition}
\newtheorem{remark}[equation]{Remark}

\renewcommand{\phi}{\varphi}

\newcommand{\I}{{\rm i}}

\renewcommand{\(}{\bigl(}
\renewcommand{\)}{\bigr)\vphantom{)}}

\newcommand{\ip}[2]{\langle#1,#2\rangle}

\newcommand{\const}{\operatorname{const}}
\newcommand{\spn}{\operatorname{span}}
\newcommand{\rank}{\operatorname{rank}}

\newcommand{\Ker}{\operatorname{Ker}}
\renewcommand{\Im}{\operatorname{Im}}

\newcommand{\One}{{1\hskip-2.5pt{\rm l}}}

\newcommand{\ga}{\gamma}

\newcommand{\de}{\delta}

\newcommand{\al}{\alpha}
\newcommand{\be}{\beta}

\newcommand{\Ec}{\mathcal E}

\newcommand{\A}{\mathcal A}
\newcommand{\B}{\mathcal B}
\newcommand{\la}{\lambda}

\newcommand{\C}{\mathbb C}

\newcommand{\dimensional}[1]{$#1$\nobreakdash-\hspace{0pt}dimensional}

\makeatletter
\renewcommand*\l@section[2]{%
  \ifnum \c@tocdepth >\z@
    \addpenalty\@secpenalty
    \addvspace{0.25em \@plus\p@}%
    \setlength\@tempdima{2.5em}%
    \begingroup
      \parindent \z@ \rightskip \@pnumwidth
      \parfillskip -\@pnumwidth
      \leavevmode \bfseries
      \advance\leftskip\@tempdima
      \hskip -\leftskip
      #1\nobreak\hfil \nobreak\hb@xt@\@pnumwidth{\hss #2}\par
    \endgroup
  \fi}
\renewcommand*\numberline[1]{\hb@xt@\@tempdima{\hfil#1\hskip1em}}
\makeatother

\begin{document}

\title{Graded algebras and subproduct systems: dimension two}

\author{Boris Tsirelson}

\date{}
\maketitle

\begin{abstract}
Objects dual to graded algebras are subproduct systems of linear spaces, a
purely algebraic counterpart of a notion introduced recently in the context of
noncommutative dynamics (Shalit and Solel \cite{SS}, Bhat and Mukherjee
\cite{BM}). A complete classification of these objects in the lowest
nontrivial dimension is given in this work, triggered by a question of Bhat
\cite{B}.
\end{abstract}

\setcounter{tocdepth}{1}
\tableofcontents

\section[Definition and theorem]
 {\raggedright Definition and theorem}
\label{sec:1}
By a subproduct system I mean in this work a discrete-time subproduct system
of two-dimensional linear spaces over $ \C $, defined as follows.

\begin{definition}\label{subproduct_system}
A \emph{subproduct system} consists of two-dimensional linear spaces $ E_t $
for $ t = 1,2,\dots $ and injective linear maps
\[
\be_{s,t} : E_{s+t} \to E_s \otimes E_t
\]
for $ s,t \in \{1,2,\dots\} $, satisfying the associativity condition:
the diagram\footnote{%
 Of course, $ \One_t : E_t \to E_t $ is the identity map.}
\[
\xymatrix{
 & E_{r+s+t} \ar[dl]_{\be_{r+s,t}} \ar[dr]^{\be_{r,s+t}}
\\
 E_{r+s} \otimes E_t \ar[dr]_{\be_{r,s}\otimes\One_t} & & E_r \otimes E_{s+t}
 \ar[dl]^{\One_r\otimes\be_{s,t}}
\\
 & E_r \otimes E_s \otimes E_t
}
\]
is commutative for all $ r,s,t \in \{1,2,\dots\} $.
\end{definition}

The dual object is the graded algebra\footnote{%
 For the definition see e.g.\ \cite[Chapter XVI, Sect. 4]{McL}.}
$ \A $ whose homogeneous component $ \A_t $ is the linear space $ E'_t $ of
all linear functions $ E_t \to \C $, and the multiplication map $ \A_s \otimes
\A_t \to \A_{s+t} $ is the dual to $ \be_{s,t} $. (And in addition, $ \A_0 =
\C $.) The associativity condition stipulated above for the subproduct system
is evidently equivalent to the usual associativity of $ \A $, and the
injectivity of $ \be_{s,t} $ --- to surjectivity of its dual. Thus, the graded
algebra $ \A $ satisfies
\begin{equation}\label{cond_for_A}
\begin{gathered}
\A_t \text{ is two-dimensional (as a linear space)} \, , \\
\text{the multiplication map } \A_s \otimes \A_t \to \A_{s+t} \text{ is surjective}
\end{gathered}
\end{equation}
for all $ s,t \in \{1,2,\dots\} $.

Here is a simple construction of a graded algebra. Given an algebra $ D $
(over $ \C $, associative, not necessarily with unit) and its automorphism $
\eta $, we let
\begin{equation}\label{the_construction}
\begin{gathered}
\B_t = D \quad \text{for } t = 1,2,\dots \, , \\
x \times_\B y = x \times_D \eta^s (y) \quad \text{for } x \in \B_s, \, y
 \in \B_t \, ;
\end{gathered}
\end{equation}
here ``$ \times_\B $'' and ``$ \times_D $'' stand for multiplication in $ \B
$ and $ D $ respectively. (And in addition, $ \B_0 = \C $.) Associativity is
easy to check: for $ x \in \B_r $, $ y \in \B_s $, $ z \in \B_t $,
\[
( x \times_\B y ) \times_\B z = x \times_D \eta^r (y) \times_D \eta^{r+s} (z)
= x \times_D \eta^r ( y \times_D \eta^s (z) ) = x \times_\B ( y \times_\B z )
\, .
\]
Assuming that $ D $ satisfies
\begin{equation}\label{cond_for_B}
\begin{gathered}
D \text{ is two-dimensional (as a linear space)} \, , \\
\text{the multiplication map } D \otimes D \to D \text{ is surjective}
\end{gathered}
\end{equation}
we get $ \B $ satisfying \eqref{cond_for_A} (see Lemma \ref{5.8}).

Surprisingly, this construction exhausts all such graded algebras.

\begin{theorem}\label{main_theorem}
For every graded algebra $ \A $ satisfying \eqref{cond_for_A} there exist an
algebra $ D $ satisfying \eqref{cond_for_B} and its automorphism $ \eta $ such
that the construction \eqref{the_construction} gives a graded algebra $ \B $
isomorphic to $ \A $.
\end{theorem}

There exist four mutually non-isomorphic algebras $ D $ satisfying
\eqref{cond_for_B}; one of them has infinitely many automorphisms, another one
has two automorphisms, the other two have only trivial automorphisms. Thus,
all graded algebras satisfying \eqref{cond_for_A} are classified, and all
subproduct systems are also classified.

However, the proof of Theorem \ref{main_theorem} given in Sect.~\ref{sec:6} is
preceded by the classification of subproduct systems given in
Sect.~\ref{sec:4}.

\section[Main lemma]
 {\raggedright Main lemma}
\label{sec:2}
Let $ L_1, L_2, L_3 $ be two-dimensional linear spaces over the field $ \C $
of complex numbers (or any other algebraically closed field of characteristic
not $ 2 $). Let two-dimensional linear subspace $ L_{12} $ of the tensor
product $ L_1 \otimes L_2 $ be given, and similarly $ L_{23} \subset L_2
\otimes L_3 $, $ \dim L_{23} = 2 $. Then $ L_{12} \otimes L_3 $ is a
\dimensional{4} subspace of the \dimensional{8} space $ L_1 \otimes L_2
\otimes L_3 $, and $ L_1 \otimes L_{23} \subset L_1 \otimes L_2 \otimes L_3 $
is another \dimensional{4} subspace. Their intersection is in general trivial
(that is, \dimensional{0}), but may also be nontrivial.

\begin{lemma}\label{1.1}
If $ ( L_{12} \otimes L_3 ) \cap ( L_1 \otimes L_{23} ) $ contains a non-zero
vector, then it contains a non-zero product vector.
\end{lemma}

In other words: \dots then there exist non-zero vectors $ x_1 \in L_1 $, $ x_2
\in L_2 $, $ x_3 \in L_3 $ such that $ x_1 \otimes x_2 \in L_{12} $ and $ x_2
\otimes x_3 \in L_{23} $. 

\begin{remark}
The conditions $ \dim L_{12} = 2 $, $ \dim L_{23} = 2 $ may be weakened: $
\dim L_{12} \ge 2 $, $ \dim L_{23} \ge 2 $. Of course, the intersection cannot
be trivial when $ \dim L_{12} + \dim L_{23} > 4 $. On the other hand, the
lemma fails when $ \dim L_{12} = 1 $, $ \dim L_{23} = 2 $ (try $ L_{23} = L_2
\otimes x_3 $).
\end{remark}

\begin{proof}
We take $ u_{12}, v_{12} \in L'_1 \otimes L'_2 $ such that $ L_{12} = \{
u_{12}, v_{12} \}^\perp $, that is, $ L_{12} = \{ z : \ip{u_{12}}z = 0 =
\ip{v_{12}}z \} $, and $ L_{23} = \{ u_{23}, v_{23} \}^\perp $. We take also
bases $ \{ u_1,v_1 \} $ in $ L'_1 $, $ \{ u_2,v_2 \} $ in $ L'_2 $, $ \{
u_3,v_3 \} $ in $ L'_3 $. Then $ L_{12} \otimes L_3 = \{ u_{12} \otimes u_3,
u_{12} \otimes v_3, v_{12} \otimes u_3, v_{12} \otimes v_3 \}^\perp $ and $
L_1 \otimes L_{23} = \{ u_1 \otimes u_{23}, u_1 \otimes v_{23}, v_1 \otimes
u_{23}, v_1 \otimes v_{23} \}^\perp $. The condition $ ( L_{12} \otimes L_3 )
\cap ( L_1 \otimes L_{23} ) \ne \{ 0 \} $ implies that these eight vectors ($
u_{12} \otimes u_3, \dots, v_1 \otimes v_{23} $) are linearly dependent. In
coordinates it means that the corresponding determinant $ 8 \times 8 $
vanishes, see \eqref{det8} below.

On the other hand, the relation $ x_1 \otimes x_2 \in L_{12} $ holds if (and
only if) $ \ip{ u_{12} }{ x_1 \otimes x_2 } = 0 = \ip{ v_{12} }{ x_1 \otimes
x_2 } $, that is, $ \ip{ U_{12} x_2 }{ x_1 } = 0 = \ip{ V_{12} x_2 }{ x_1 } $,
where linear maps $ U_{12}, V_{12} : L_2 \to L'_1 $ correspond to the given $
u_{12}, v_{12} \in L'_1 \otimes L'_2 $ in the sense that $ \ip{ U_{12} x_2 }{
x_1 } = \ip{ u_{12} }{ x_1 \otimes x_2 } $ for all $ x_1 \in L_1 $ (and the
same for $ V $). Existence of such $ x_1 \ne 0 $ (for a given $ x_2 $) holds
if (and only if) vectors $ U_{12} x_2 $ and $ V_{12} x_2 $ are linearly
dependent. In coordinates it means that the corresponding determinant $ 2
\times 2 $ vanishes, see \eqref{det2} below. The determinant is a quadratic
form of $ x_2 $.

Similarly, a vector $ x_3 \ne 0 $ such that $ x_2 \otimes x_3 \in L_{23} $
exists if and only if another quadratic form vanishes at $ x_2 $. We have to
prove that these two quadratic forms vanish simultaneously on some non-zero
vector of the two-dimensional space $ L_2 $. A sufficient (and necessary)
condition is that the resultant of these two quadratic forms
vanishes.\footnote{%
 See e.g.\ \cite[Chapter IV, Sect. 27]{Wa}.}
The resultant is the determinant of a $ 4 \times 4 $ matrix whose elements are
quadratic forms of $ u_{12}, v_{12}, u_{23}, v_{23} $.

We turn to coordinates;
\begin{align*}
u_{12} &= a u_1 \otimes u_2 + b u_1 \otimes v_2 + c v_1 \otimes u_2 + d v_1
 \otimes v_2 \, , \\
v_{12} &= e u_1 \otimes u_2 + f u_1 \otimes v_2 + g v_1 \otimes u_2 + h v_1
 \otimes v_2 \, , \\
u_{23} &= A u_2 \otimes u_3 + B v_2 \otimes u_3 + C u_2 \otimes v_3 + D v_2
 \otimes v_3 \, , \\
v_{23} &= E u_2 \otimes u_3 + F v_2 \otimes u_3 + G u_2 \otimes v_3 + H v_2
 \otimes v_3
\end{align*}
for some $ a,b,c,d,e,f,g,h,A,B,C,D,E,F,G,H \in \C $. Using the basis
\begin{multline*}
\( u_1 \otimes u_2 \otimes u_3, u_1 \otimes v_2 \otimes u_3, v_1 \otimes u_2
 \otimes u_3, v_1 \otimes v_2 \otimes u_3, \\
u_1 \otimes u_2 \otimes v_3, u_1 \otimes v_2 \otimes v_3, v_1 \otimes u_2
 \otimes v_3, v_1 \otimes v_2 \otimes v_3 \)
\end{multline*}
in $ L'_1 \otimes L'_2 \otimes L'_3 $ we see that the following determinant
(denote it $ D_8 $) vanishes:
\begin{equation}\label{det8}
\begin{array}{lccccccccccc}
 & & & uuu & uvu & vuu & vvu & uuv & uvv & vuv & vvv & \\
 & & & & & & & & & & & \\
u_{12} \otimes u_3 & & \vline & a & b & c & d & 0 & 0 & 0 & 0 & \vline \\
v_{12} \otimes u_3 & & \vline & e & f & g & h & 0 & 0 & 0 & 0 & \vline \\
u_{12} \otimes v_3 & & \vline & 0 & 0 & 0 & 0 & a & b & c & d & \vline \\
v_{12} \otimes v_3 & & \vline & 0 & 0 & 0 & 0 & e & f & g & h & \vline \\
u_1 \otimes u_{23} & & \vline & A & B & 0 & 0 & C & D & 0 & 0 & \vline \\
u_1 \otimes v_{23} & & \vline & E & F & 0 & 0 & G & H & 0 & 0 & \vline \\
v_1 \otimes u_{23} & & \vline & 0 & 0 & A & B & 0 & 0 & C & D & \vline \\
v_1 \otimes v_{23} & & \vline & 0 & 0 & E & F & 0 & 0 & G & H & \vline
\end{array}
\end{equation}
(The additional row and column outside the determinant are just comments.)

Vectors
\begin{align*}
U_{12} x_2 &= \ip{ au_2+bv_2 }{ x_2 } u_1 + \ip{ cu_2+dv_2 }{ x_2 } v_1 \, ,
 \\
V_{12} x_2 &= \ip{ eu_2+fv_2 }{ x_2 } u_1 + \ip{ gu_2+hv_2 }{ x_2 } v_1
\end{align*}
are linearly dependent when the determinant
\begin{equation}\label{det2}
\begin{vmatrix}
\ip{ au_2+bv_2 }{ x_2 } & \ip{ cu_2+dv_2 }{ x_2 } \\
\ip{ eu_2+fv_2 }{ x_2 } & \ip{ gu_2+hv_2 }{ x_2 }
\end{vmatrix}
\end{equation}
vanishes, which is the condition for existence of $ x_1 $. Similarly, $ x_3 $
exists when the determinant
\[
\begin{vmatrix}
\ip{ Au_2+Bv_2 }{ x_2 } & \ip{ Cu_2+Dv_2 }{ x_2 } \\
\ip{ Eu_2+Fv_2 }{ x_2 } & \ip{ Gu_2+Hv_2 }{ x_2 }
\end{vmatrix}
\]
vanishes.

We rewrite \eqref{det2} as
\[
p \ip{ u_2 }{ x_2 }^2 + q \ip{ u_2 }{ x_2 } \ip{ v_2 }{ x_2 } + r \ip{ v_2 }{ x_2 }^2
\]
where
\begin{equation}\label{2.5}
p = \begin{vmatrix} a & c \\ e & g \end{vmatrix} \, , \quad
q = \begin{vmatrix} a & d \\ e & h \end{vmatrix} + \begin{vmatrix} b & c \\ f
  & g \end{vmatrix} \, , \quad
r = \begin{vmatrix} b & d \\ f & h \end{vmatrix}
\end{equation}
are the coefficients of the quadratic form \eqref{det2}. Similarly, the other
quadratic form has the coefficients
\begin{equation}\label{2.6}
P = \begin{vmatrix} A & C \\ E & G \end{vmatrix} \, , \quad
Q = \begin{vmatrix} A & D \\ E & H \end{vmatrix} + \begin{vmatrix} B & C \\ F
  & G \end{vmatrix} \, , \quad
R = \begin{vmatrix} B & D \\ F & H \end{vmatrix} \, .
\end{equation}
The resultant of these two quadratic forms is
\begin{equation}\label{2.7}
D_4 = \begin{vmatrix}
p & q & r & 0 \\ 0 & p & q & r \\ P & Q & R & 0 \\ 0 & P & Q & R
\end{vmatrix}
\end{equation}
Now we may forget all the linear spaces and retain only the two determinants
(as functions of $ a,b,c,d,e,f,g,h,A,B,C,D,E,F,G,H $): $ D_8 $ defined by
\eqref{det8}, and $ D_4 $ defined by \eqref{2.5}--\eqref{2.7}. It is
sufficient to prove that $ D_4 = 0 $ whenever $ D_8 = 0 $.

Surprisingly,
\[
D_8 + D_4 = 0
\]
for all $ a,b,c,d,e,f,g,h,A,B,C,D,E,F,G,H $. This equality can be checked
using a computer algebra system (I did so); it is also feasible (but tedious)
to check it manually, see Appendix.
\end{proof}

\section[Two tensor factors]
 {\raggedright Two tensor factors}
\label{sec:3}
Let $ L_1, L_2 $ be two-dimensional linear spaces over $ \C $, and $ L_{12}
\subset L_1 \otimes L_2 $, $ \dim L_{12} = 2 $.

The space $ L_1 \otimes L_2 $ carries a quadratic form $ A : L_1 \otimes L_2
\to \C $ that vanishes on all product vectors $ x_1 \otimes x_2 $ and only
these vectors. The form $ A $ is unique up to a coefficient. In terms of a
basis $ \{x_1,y_1\} $ of $ L_1 $ and a basis $ \{x_2,y_2\} $ of $ L_2 $,
\[
A ( \al x_1 \otimes x_2 + \be x_1 \otimes y_2 + \ga y_1 \otimes x_2 + \de y_1
\otimes y_2 ) = \const \cdot \begin{vmatrix} \al & \be \\ \ga & \de
\end{vmatrix} \, .
\]
The rank of the form $ A $ is equal to $ 4 $. The restriction of $ A $ to $ L_{12}
$ may be of rank $ 0 $, $ 1 $, or $ 2 $. Accordingly we say that $ L_{12} $ is of
rank $ 0 $, $ 1 $, or $ 2 $. For some basis $ \{\xi,\psi\} $ of $ L_{12} $ we
have
\begin{equation}\label{*}
A(a\xi+b\psi) = \begin{cases}
a^2+b^2 &\text{if } \rank L_{12} = 2, \\
a^2 &\text{if } \rank L_{12} = 1, \\
0 &\text{if } \rank L_{12} = 0
\end{cases}
\end{equation}
for all $ a,b \in \C $.

\begin{lemma}\label{2.2}
If $ \rank L_{12} = 2 $ then there exist bases $ \{x_1,y_1\}
$ of $ L_1 $ and $ \{x_2,y_2\} $ of $ L_2 $ such that\footnote{%
 Of course, ``$\spn$'' stands for the linear space spanned by given vectors.}
\[
L_{12} = \spn \{ x_1 \otimes x_2, \, y_1 \otimes y_2 \} \, .
\]
All product vectors in $ L_{12} $ are of the form $ \la x_1 \otimes x_2 $ or $
\la y_1 \otimes y_2 $ ($ \la \in\C $).
\end{lemma}

\begin{proof}
By \eqref{*}, $ A(\xi\pm\I\psi) = 0 $, therefore $ \xi+\I\psi = x_1 \otimes
x_2 $ and $ \xi-\I\psi = y_1 \otimes y_2 $ for some nonzero vectors $ x_1,y_1
\in L_1 $, $ x_2,y_2 \in L_2 $. Vectors $ x_1, y_1 $ are linearly independent,
since otherwise $ 2\xi = x_1 \otimes x_2 + y_1 \otimes y_2 = x_1 \otimes ( x_2
+ c y_2 ) $ would imply $ A(\xi) = 0 $. Similarly, $ x_2, y_2 $ are linearly
independent. It remains to note that all product vectors in $ L_{12} $ are of
the form $ \la (\xi\pm\I\psi) $ by \eqref{*}.
\end{proof}

\begin{lemma}\label{2.3}
If $ \rank L_{12} = 1 $ then there exist bases $ \{x_1,y_1\}
$ of $ L_1 $ and $ \{x_2,y_2\} $ of $ L_2 $ such that
\[
L_{12} = \spn \{ x_1 \otimes x_2, \, y_1 \otimes x_2 + x_1 \otimes y_2 \} \, .
\]
All product vectors in $ L_{12} $ are of the form $ \la x_1 \otimes x_2 $ ($
\la \in\C $).
\end{lemma}

\begin{proof}
By \eqref{*}, $ A(\xi\pm\psi) = 1 $ and $ A(\psi)=0 $. Thus, $
\psi $ is a product vector, and we can choose bases $ \{x_1,y_1\} $ of $ L_1 $
and $ \{x_2,y_2\} $ of $ L_2 $ such that $ \psi = x_1 \otimes x_2 $. Taking $
\al, \be, \ga, \de \in \C $ such that $ \xi = \al x_1 \otimes x_2 + \be x_1
\otimes y_2 + \ga y_1 \otimes x_2 + \de y_1 \otimes y_2 $ we get
\begin{multline*}
1 = A(\xi\pm\psi) = A \( (\al\pm1) x_1 \otimes x_2 + \be x_1 \otimes y_2 + \ga
 y_1 \otimes x_2 + \de y_1 \otimes y_2 \) = \\
= \begin{vmatrix} \al\pm1 & \be \\ \ga & \de \end{vmatrix} = \begin{vmatrix}
 \al & \be \\ \ga & \de \end{vmatrix} \pm \de \, .
\end{multline*}
Thus, $ \de = 0 $, $ \be \ne 0 $, $ \ga \ne 0 $. We have $ L_{12} =
\spn(\psi,\xi) = \spn(\psi,\xi-\al\psi) = \spn ( x_1 \otimes x_2, \, \be x_1
\otimes y_2 + \ga y_1 \otimes x_2 ) = \spn \( x_1 \otimes x_2, \, x_1 \otimes (
\be y_2 ) + ( \ga y_1 ) \otimes x_2 \) $. It remains to note that all product
vectors in $ L_{12} $ are of the form $ \la \psi $ by \eqref{*}.
\end{proof}

\begin{lemma}\label{2.4}
If $ \rank L_{12} = 0 $ then there exist bases $ \{x_1,y_1\}
$ of $ L_1 $ and $ \{x_2,y_2\} $ of $ L_2 $ such that
\[
\text{either } L_{12} = \spn \{ x_1 \otimes x_2, \, y_1 \otimes x_2 \} \text{
  or } L_{12} = \spn \{ x_1 \otimes x_2, \, x_1 \otimes y_2 \} \, .
\]
That is, $ L_{12} $ is either $ L_1 \otimes x_2 $ or $ x_1 \otimes L_2 $. All
vectors in $ L_{12} $ are product vectors.
\end{lemma}

\begin{proof}
We take bases $ \{x_1,y_1\} $ of $ L_1 $ and $ \{x_2,y_2\} $ of $ L_2 $ such
that $ x_1 \otimes x_2 \in L_{12} $ and consider the dual (biorthogonal) bases
$ \{x'_1,y'_1\} $ of $ L'_1 $ and $ \{x'_2,y'_2\} $ of $ L'_2 $. By \eqref{*},
\[
\ip{ x'_1 \otimes x'_2 }{ \psi } \ip{ y'_1 \otimes y'_2 }{ \psi } -
\ip{ x'_1 \otimes y'_2 }{ \psi } \ip{ y'_1 \otimes x'_2 }{ \psi } = 0
\]
for all $ \psi \in L_{12} $. The same holds for the vector $ \psi - \ip{ x'_1
\otimes x'_2 }{ \psi } x_1 \otimes x_2 $, therefore
\[
\ip{ x'_1 \otimes y'_2 }{ \psi } \ip{ y'_1 \otimes x'_2 }{ \psi } = 0 \quad
\text{and} \quad 
\ip{ x'_1 \otimes x'_2 }{ \psi } \ip{ y'_1 \otimes y'_2 }{ \psi } = 0
\]
for all $ \psi \in L_{12} $. Taking into account that $ \ip{ x'_1 \otimes x'_2
}{ \psi } $ is not always $ 0 $ we conclude that $ \ip{ y'_1 \otimes y'_2 }{
\psi } = 0 $ for all $ \psi \in L_{12} $, and either $ \ip{ x'_1 \otimes y'_2
}{ \psi } = 0 $ for all $ \psi \in L_{12} $, or $ \ip{ y'_1 \otimes x'_2 }{
\psi } = 0 $  for all $ \psi \in L_{12} $. The former implies $ L_{12} = L_1
\otimes x_2 $; the latter implies $ L_{12} = x_1 \otimes L_2 $.
\end{proof}

\section[Three tensor factors]
 {\raggedright Three tensor factors}
\label{sec:4}
\begin{center}\textsc{different factors}\end{center}

Assume that two-dimensional linear spaces
\pagebreak[1]
$ L_1 $, $ L_2 $, $ L_3 $, $ L_{12}
$, $ L_{23} $ and $ L_{123} $ satisfy
\begin{equation}\label{305}
\begin{gathered}
L_{12} \subset L_1 \otimes L_2 \, , \quad L_{23} \subset L_2 \otimes L_3 \, ,
 \quad L_{123} \subset L_1 \otimes L_2 \otimes L_3 \, , \\
L_{123} \subset ( L_{12} \otimes L_3 ) \cap ( L_1 \otimes L_{23} ) \, .
\end{gathered}
\end{equation}

According to Sect.~\ref{sec:3}, the subspace $ L_{12} $ of $ L_1 \otimes L_2 $
may be of rank $ 0 $, $ 1 $, or $ 2 $ and similarly $ L_{23} \subset L_2
\otimes L_3 $ may be of rank $ 0 $, $ 1 $, or $ 2 $.

\begin{lemma}\label{3.1}
If $ \rank L_{12} = 2 $ and $ \rank L_{23} \ne 0 $ then $ \rank L_{23} = 2 $
and there exist bases $ \{x_1,y_1\} $ of $ L_1 $, $ \{x_2,y_2\} $ of $ L_2 $
and $ \{x_3,y_3\} $ of $ L_3 $ such that
\[
L_{123} = \spn \{ x_1 \otimes x_2 \otimes x_3, y_1 \otimes y_2 \otimes y_3 \}  \, .
\]
\end{lemma}

\begin{proof}
The claim ``$ \rank L_{23} = 2 $'' follows from the formula for $ L_{123} $,
proved below. Lemma \ref{1.1} gives non-zero vectors $ z_1 \in L_1 $, $ z_2
\in L_2 $ and $ z_3 \in L_3 $ such that $ z_1 \otimes z_2 \otimes z_3 \in
L_{123} $. On the other hand, Lemma \ref{2.2} gives bases $ \{x_1,y_1\} $ of $
L_1 $ and $ \{x_2,y_2\} $ of $ L_2 $ such that $ L_{12} = \spn \{ x_1 \otimes
x_2, \, y_1 \otimes y_2 \} $ and states that $ z_1 \otimes z_2 $ is either $
\la x_1 \otimes x_2 $ or $ \la y_1 \otimes y_2 $. Without loss of generality,
$ z_1 \otimes z_2 = x_1 \otimes x_2 $. We choose a basis $ \{x_3,y_3\} $ of $
L_3 $ such that $ z_3 = x_3 $; thus $ x_1 \otimes x_2 \otimes x_3 \in L_{123}
$.

We have $ L_{123} = \spn \{ x_1 \otimes x_2 \otimes x_3, \, \psi \} $ for some $
\psi = \al x_1 \otimes x_2 \otimes y_3 + \be y_1 \otimes y_2 \otimes x_3 + \ga
y_1 \otimes y_2 \otimes y_3 $; combinations $ x_1 \otimes y_2 $ and $ y_1
\otimes x_2 $ cannot appear, since $ \psi \in L_{12} \otimes L_3 $. The space
$ L_{23} $ contains $ x_2 \otimes x_3 $, since $ x_1 \otimes x_2 \otimes x_3
\in L_{123} \subset L_1 \otimes L_{23} $. Also, $ \al x_2 \otimes y_3 \in
L_{23} $ and $ \be y_2 \otimes x_3 + \ga y_2 \otimes y_3 \in L_{23} $, since $
\psi \in L_1 \otimes L_{23} $. However, $ x_2 \otimes y_3 $ cannot belong to $
L_{23} $, since otherwise $ L_{23} \supset x_2 \otimes L_3 $ in contradiction
to ``$ \rank L_{23} \ne 0 $''. Similarly, $ y_2 \otimes x_3 \notin L_{23} $, since
otherwise $ L_{23} \supset L_2 \otimes x_3 $. Thus, $ \al = 0 $, $ \ga \ne 0
$ (it cannot happen that $ \be=\ga=0 $), and $ L_{123} = \spn \( x_1 \otimes
x_2 \otimes x_3, \, y_1 \otimes y_2 \otimes (\be x_3 + \ga y_3) \) $; it
remains to redefine $ y_3 $ as $ \be x_3 + \ga y_3 $.
\end{proof}

\begin{lemma}\label{3.2}
If $ \rank L_{12} = 1 $ and $ \rank L_{23} \ne 0 $ then $ \rank L_{23} = 1 $
and there exist bases $ \{x_1,y_1\} $ of $ L_1 $, $ \{x_2,y_2\} $ of $ L_2 $
and $ \{x_3,y_3\} $ of $ L_3 $ such that
\[
L_{123} = \spn \{ x_1 \otimes x_2 \otimes x_3, y_1 \otimes x_2 \otimes x_3 + x_1
\otimes y_2 \otimes x_3 + x_1 \otimes x_2 \otimes y_3 \} \, .
\]
\end{lemma}

\begin{proof}
The claim ``$ \rank L_{23} = 1 $'' follows from the formula for $ L_{123} $,
proved below. Similarly to the proof of \ref{3.1}, using \ref{1.1} and
\ref{2.3} we get bases $ \{x_1,y_1\} $, $ \{x_2,y_2\} $ and $ \{x_3,y_3\} $
such that $ L_{12} = \spn \{ x_1 \otimes x_2, \, y_1 \otimes x_2 + x_1 \otimes
y_2 \} $ and  $ x_1 \otimes x_2 \otimes x_3 \in L_{123} $.

We have $ L_{123} = \spn \{ x_1 \otimes x_2 \otimes x_3, \, \psi \} $ for some
$ \psi = \al x_1 \otimes x_2 \otimes y_3 + ( y_1 \otimes x_2 + x_1 \otimes y_2
) \otimes ( \be x_3 + \ga y_3 ) = x_1 \otimes \( \al x_2 \otimes y_3 + y_2
\otimes ( \be x_3 + \ga y_3 ) \) + y_1 \otimes x_2 \otimes ( \be x_3 + \ga y_3
) $. Thus, $ x_2 \otimes x_3 \in L_{23} $, $ \al x_2 \otimes y_3 + y_2 \otimes
( \be x_3 + \ga y_3 ) \in L_{23} $ and $ x_2 \otimes ( \be x_3 + \ga y_3 ) \in
L_{23} $. If $ \ga \ne 0 $ then $ L_{23} \supset \spn \{ x_2 \otimes x_3, \,
x_2 \otimes ( \be x_3 + \ga y_3 ) \} = \spn \{ x_2 \otimes x_3, \, x_2 \otimes
y_3 \} = x_2 \otimes L_3 $ in contradiction to ``$ \rank L_{23} \ne 0
$''. Thus, $ \ga = 0 $, and $ L_{23} \supset \spn \{ x_2 \otimes x_3, \, \al
x_2 \otimes y_3 + \be y_3 \otimes x_3 \} $. However, $ x_2 \otimes y_3 \notin
L_{23} $ and $ y_2 \otimes x_3 \notin L_{23} $, thus $ \al \ne 0 $, $ \be \ne
0 $, and $ L_{123} = \spn \{ x_1 \otimes x_2 \otimes x_3, (\be y_1) \otimes
x_2 \otimes x_3 + x_1 \otimes (\be y_2) \otimes x_3 + x_1 \otimes x_2 \otimes
(\al y_3) \} $.
\end{proof}

It is more difficult to classify the cases where at least one rank
vanishes. Here are some possibilities:
\begin{align*}
& \spn \{ x_1 \otimes x_2 \otimes x_3, \, y_1 \otimes y_2 \otimes x_3 \} \, ; \\
& \spn \{ x_1 \otimes x_2 \otimes x_3, \, x_1 \otimes y_2 \otimes x_3 + y_1
 \otimes x_2 \otimes x_3 \} \, ; \\
& \spn \{ x_1 \otimes x_2 \otimes x_3, \, x_1 \otimes y_2 \otimes x_3 \} \, .
\end{align*}

\begin{center}\textsc{identical factors}\end{center}

Assume that two-dimensional linear spaces $ E_1, E_2, E_3 $ satisfy
\begin{equation}\label{3.25}
\begin{gathered}
E_2 \subset E_1 \otimes E_1 \, , \quad E_3 \subset E_1 \otimes E_1 \otimes E_1
 \, , \\
E_3 \subset ( E_2 \otimes E_1 ) \cap ( E_1 \otimes E_2 ) \, .
\end{gathered}
\end{equation}

\begin{lemma}\label{3.3}
If $ \rank E_2 = 2 $ then there exists a basis $ \{x,y\} $ of $ E_1 $ such
that either
\[
E_3 = \spn \{ x \otimes x \otimes x, y \otimes y \otimes y \}
\]
or
\[
E_3 = \spn \{ x \otimes y \otimes x, y \otimes x \otimes y \} \, .
\]
\end{lemma}

\begin{proof}
Lemma \ref{3.1} applied to $ L_1 = L_2 = L_3 = E_1 $, $ L_{12} = L_{23} = E_2
$ and $ L_{123} = E_3 $ gives three bases $ \{x_1,y_1\} $, $ \{x_2,y_2\} $ and
$ \{x_3,y_3\} $ of $ E_1 $ such that $ E_3 = \spn \{ x_1 \otimes x_2 \otimes
x_3, y_1 \otimes y_2 \otimes y_3 \} $. In combination with the relation $ E_3
\subset E_2 \otimes E_1 $ it implies $ x_1 \otimes x_2 \in E_2 $ and $ y_1
\otimes y_2 \in E_2 $, thus $ E_2 = \spn \{ x_1 \otimes x_2, y_1 \otimes y_2
\} $. Similarly, $ E_2 = \spn \{ x_2 \otimes x_3, y_2 \otimes y_3 \} $. Each
product vector in $ E_2 $ is either $ \la x_1 \otimes x_2 $ or $ \la y_1
\otimes y_2 $.

Case 1: $ x_2 \otimes x_3 = \la x_1 \otimes x_2 $ and $ y_2 \otimes y_3 = \mu
y_1 \otimes y_2 $. Then $ x_1, x_2, x_3 $ are collinear; $ x_1 = \la_1 x $, $
x_2 = \la_2 x $, $ x_3 = \la_3 x $ for some $ x $. Similarly, $ y_1 = \mu_1 y
$, $ y_2 = \mu_2 y $, $ y_3 = \mu_3 y $. Thus, $ E_3 = \spn \{ x \otimes x
\otimes x, y \otimes y \otimes y \} $.

Case 2: $ x_2 \otimes x_3 = \la y_1 \otimes y_2 $ and $ y_2 \otimes y_3 = \mu
x_1 \otimes x_2 $. Then $ x_1, y_2, x_3 $ are collinear (to some $ x $), $
y_1, x_2, y_3 $ are collinear (to some $ y $), and we get $ E_3 = \spn \{ x
\otimes y \otimes x, y \otimes x \otimes y \} $.
\end{proof}

\begin{lemma}\label{3.4}
If $ \rank E_2 = 1 $ then there exists a basis $ \{x,y\} $ of $ E_1 $ and a
number $ \la \in \C $, $ \la \ne 0 $, such
that
\[
E_3 = \spn \{ x \otimes x \otimes x, \, y \otimes x \otimes x + \la x \otimes
y \otimes x + \la^2 x \otimes x \otimes y \} \, .
\]
\end{lemma}

\begin{proof}
Similarly to the proof of Lemma \ref{3.3}, first, Lemma \ref{3.2} gives three
bases $ \{x_1,y_1\} $, $ \{x_2,y_2\} $ and $ \{x_3,y_3\} $ of $ E_1 $ such
that $ E_3 = \spn \{ x_1 \otimes x_2 \otimes x_3, \, y_1 \otimes x_2 \otimes
x_3 + x_1 \otimes y_2 \otimes x_3 + x_1 \otimes x_2 \otimes y_3 \} $; second,
$ E_2 $ contains the vectors $ x_1 \otimes x_2 $, $ y_1 \otimes x_2 + x_1
\otimes y_2 $, $ x_2 \otimes x_3 $, $ y_2 \otimes x_3 + x_2 \otimes y_3 $; and
third, $ x_1 \otimes x_2 $ and $ x_2 \otimes x_3 $ are collinear, thus, $ x_1,
x_2, x_3 $ are collinear to some $ x $.

Replacing $ y_k $ with $ y_k + c_k x $ for appropriate $ c_k $ we ensure that
$ y_1, y_2, y_3 $ are collinear to some $ y $. Now $ E_3 = \spn \{ x \otimes x
\otimes x, \, a_1 y \otimes x \otimes x + a_2 x \otimes y \otimes x + a_3 x
\otimes x \otimes y \} $ and $ E_2 = \spn \{ x \otimes x, \, a_1 y \otimes x +
a_2 x \otimes y, a_2 y \otimes x + a_3 x \otimes y \} $. The last two vectors
must be collinear, thus $ \frac{ a_2 }{ a_1 } = \frac{ a_3 }{ a_2 } = \la $
for some $ \la $.
\end{proof}

\begin{lemma}\label{3.5}
If $ \rank E_2 = 0 $ then there exists a basis $ \{x,y\} $ of $ E_1 $ such
that either
\[
E_3 = \spn \{ x \otimes x \otimes x, y \otimes x \otimes x \}
\]
or
\[
E_3 = \spn \{ x \otimes x \otimes x, x \otimes x \otimes y \} \, .
\]
\end{lemma}

That is, $ E_3 $ is either $ E_1 \otimes x \otimes x $ or $ x \otimes x
\otimes E_1 $.

\begin{proof}
By Lemma \ref{2.4}, $ E_2 $ is either $ E_1 \otimes x $ or $ x \otimes E_1 $
for some $ x $. Assume $ E_2 = E_1 \otimes x $ (the other case is similar). We
have $ E_3 \subset ( E_2 \otimes E_1 ) \cap ( E_1 \otimes E_2 ) = ( E_1
\otimes x \otimes E_1 ) \cap ( E_1 \otimes E_1 \otimes x ) = E_1 \otimes x
\otimes x $, therefore $ E_3 = E_1 \otimes x \otimes x $.
\end{proof}

\begin{sloppypar}
Given a triple $ (E_1,E_2,E_3) $ satisfying \eqref{3.25} and another such triple
$ (E'_1,E'_2,E'_3) $, we define an isomorphism between the two triples as an
invertible linear map $ \theta : E_1 \to E'_1 $ such that $ ( \theta \otimes
\theta ) (E_2) = E'_2 $ and $ ( \theta \otimes \theta \otimes \theta ) (E_3) =
E'_3 $.
\end{sloppypar}

Lemmas \ref{3.3}, \ref{3.4}, \ref{3.5} show that the following examples
exhaust all triples up to isomorphism.

\begin{example}\label{3.6}
$ E_1 = \C^2 $, $ E_2 = \spn \{ e_1 \otimes e_1, \, e_2 \otimes e_2 \} $, $
E_3 = \spn \{ e_1 \otimes e_1 \otimes e_1, \, e_2 \otimes e_2 \otimes e_2 \}
$; here $ e_1 = (1,0) $ and $ e_2 = (0,1) $ are the standard basis vectors in
$ \C^2 $.
\end{example}

\begin{example}\label{3.7}
$ E_1 = \C^2 $, $ E_2 = \spn \{ e_1 \otimes e_2, \, e_2 \otimes e_1 \} $, $
E_3 = \spn \{ e_1 \otimes e_2 \otimes e_1, \, e_2 \otimes e_1 \otimes e_2 \}
$.
\end{example}

\begin{example}\label{3.8}
Given $ \la \in \C $, $ \la \ne 0 $, we let $ E_1 = \C^2 $, $ E_2 = \spn \{
e_1 \otimes e_1, \, e_2 \otimes e_1 + \la e_1 \otimes e_2 \} $, $ E_3 = \spn
\{ e_1 \otimes e_1 \otimes e_1, \, e_2 \otimes e_1 \otimes e_1 + \la e_1
\otimes e_2 \otimes e_1 + \la^2 e_1 \otimes e_1 \otimes e_2 \} $.
\end{example}

\begin{example}\label{3.9}
$ E_1 = \C^2 $, $ E_2 = \spn \{ e_1 \otimes e_1, \, e_2 \otimes e_1 \} $, $
E_3 = \spn \{ e_1 \otimes e_1 \otimes e_1, \, e_2 \otimes e_1 \otimes e_1 \}
$.
\end{example}

\begin{example}\label{3.10}
$ E_1 = \C^2 $, $ E_2 = \spn \{ e_1 \otimes e_1, \, e_1 \otimes e_2 \} $, $
E_3 = \spn \{ e_1 \otimes e_1 \otimes e_1, \, e_1 \otimes e_1 \otimes e_2 \}
$.
\end{example}

\begin{lemma}\label{3.13}
The examples given above are mutually non-isomorphic.
\end{lemma}

\begin{proof}
First,
\[
\rank E_2 = \begin{cases}
 2 &\text{for \ref{3.6} and \ref{3.7}},\\
 1 &\text{for \ref{3.8}},\\
 0 &\text{for \ref{3.9} and \ref{3.10}}.
\end{cases}
\]
Examples \ref{3.6} and \ref{3.7} are non-isomorphic, since the exchange
operator ($ x \otimes y \mapsto y \otimes x $) on $ E_1 \otimes E_1 $ acts on
$ E_2 $ trivially for \ref{3.6} but not \ref{3.7}.

Examples \ref{3.9} and \ref{3.10} are non-isomorphic, since $ E_2 $ is of the
form $ E_1 \otimes x $ for \ref{3.9} but not \ref{3.10}.

It remains to prove that Example \ref{3.8} gives non-isomorphic triples for
different $ \la $. Assume the contrary: $ \theta : \C^2 \to \C^2 $ satisfies
\[
( \theta \otimes \theta ) \spn \{ e_1 \otimes e_1, \, e_2 \otimes e_1 + \la
e_1 \otimes e_2 \} = \spn \{ e_1 \otimes e_1, \, e_2 \otimes e_1 + \mu e_1
\otimes e_2 \}
\]
(and one more condition that we do not need now), that is,
\[
\spn \{ x \otimes x, \, y \otimes x + \la x \otimes y \} = \spn \{ e_1 \otimes
e_1, \, e_2 \otimes e_1 + \mu e_1 \otimes e_2 \}
\]
where $ x = \theta e_1 $ and $ y = \theta e_2 $. All product vectors of the
first space are collinear to $ x \otimes x $, of the second --- to $ e_1
\otimes e_1 $; thus, $ x \otimes x $ is collinear to $ e_1 \otimes e_1 $, and
$ x $ is collinear to $ e_1 $. Multiplying $ y $ by an appropriate scalar we
get $ y \otimes e_1 + \la e_1 \otimes y = e_2 \otimes e_1 + \mu e_1 \otimes
e_2 + c e_1 \otimes e_1 $ for some $ c $. Taking $ a,b $ such that $ y = a e_1
+ b e_2 $ we get $ a e_1 \otimes e_1 + b e_2 \otimes e_1 + \la a e_1 \otimes
e_1 + \la b e_1 \otimes e_2 = e_2 \otimes e_1 + \mu e_1 \otimes e_2 + c e_1
\otimes e_1 $, which implies $ b = 1 $ and $ \la b = \mu $.
\end{proof}

\begin{proposition}\label{3.14}
For every triple $ (E_1,E_2,E_3) $ satisfying \eqref{3.25} there exist a
two-dimensional algebra $ D $ satisfying \eqref{cond_for_B} and its
automorphism $ \eta $ such that the triple $ (E_1,E_2,E_3) $ is isomorphic to
the triple $ (F_1,F_2,F_3) $ defined as follows:

$ F_1 $ consists of all linear functionals on $ D $;

$ F_2 $ consists of all linear functionals on $ D \otimes D $ of the form $
x \otimes y \mapsto f \(x \eta (y) \) $ where $ f $ runs over $ F_1 $;

$ F_3 $ consists of all linear functionals on $ D \otimes D \otimes D $ of
the form $ x \otimes y \otimes z \mapsto f \(x \eta (y) \eta^2 (z) \) $ where
$ f $ runs over $ F_1 $.
\end{proposition}

\begin{proof}
It is sufficient to check the claim for Examples \ref{3.6}--\ref{3.10}.

Case \ref{3.6}. We take $ D = \C^2 $ with the pointwise multiplication, that
is, $ (a_1,a_2) (b_1,b_2) = (a_1 b_1, a_2 b_2) $, and the trivial automorphism
$ \eta $, that is, $ \eta (a_1,a_2) = (a_1,a_2) $. In terms of the coordinate
functionals $ e_1, e_2 $ on $ D $ the multiplication is $ e_1(xy) = e_1(x)
e_1(y) $, $ e_2(xy) = e_2(x) e_2(y) $. Thus, $ F_2 = \spn \{ e_1 \otimes e_1,
\, e_2 \otimes e_2 \} $. A similar argument applies to $ F_3 $.

Case \ref{3.7}. We take $ D $ and $ e_1, e_2 $ as before, but $ \eta $ swaps
the coordinates, that is, $ \eta(a_1,a_2) = (a_2,a_1) $. Then $ e_1 \(x \eta
(y)\) = e_1(x) e_2(y) $, $ e_2\(x \eta (y)\) = e_2(x) e_1(y) $. Thus, $ F_2 =
\spn \{ e_1 \otimes e_2, \, e_2 \otimes e_1 \} $. A similar argument applies
to $ F_3 $.

\begin{sloppypar}
Case \ref{3.8}. We take $ D = \C^2 $ with the multiplication $ (a_1,a_2)
(b_1,b_2) = ( a_1 b_1, a_2 b_1 + a_1 b_2 ) $. Associativity is easy to check; 
$ (a_1,a_2) (b_1,b_2) (c_1,c_2) = ( a_1 b_1 c_1, a_2 b_1 c_1 + a_1 b_2 c_1 +
a_1 b_1 c_2 ) $. We also take the automorphism $ \eta : (a_1,a_2) \mapsto
(a_1,\la a_2) $. Then $ e_1 \(x \eta (y)\) = e_1(x) e_1(y) $, $ e_2\(x(\eta
y)\) = e_2(x) e_1(y) + \la e_1(x) e_2(y) $. Thus, $ F_2 = \spn \{ e_1 \otimes
e_2, \, e_2 \otimes e_1 + \la e_1 \otimes e_2 \} $. Similarly, $ e_1 \(x \eta
(y) \eta^2 (z)\) = e_1(x) e_1(y) e_1(z) $, $ e_2\(x \eta (y) \eta^2 (z)\) = e_2(x)
e_1(y) e_1(z) + \la e_1(x) e_2(y) e_1(z) + \la^2 e_1(x) e_1(y) e_2(z) $. Thus,
$ F_3 = \spn \{ e_1 \otimes e_1 \otimes e_1, \, e_2 \otimes e_1 \otimes e_1 +
\la e_1 \otimes e_2 \otimes e_1 + \la^2 e_1 \otimes e_1 \otimes e_2 \} $.
\end{sloppypar}

Case \ref{3.9}. We take $ D = \C^2 $ with the multiplication $ (a_1,a_2)
(b_1,b_2) = ( a_1 b_1, a_2 b_1 ) $, and the trivial
automorphism. Associativity is easy to check; $ (a_1,a_2) (b_1,b_2) (c_1,c_2)
= ( a_1 b_1 c_1, a_2 b_1 c_1 ) $. We have $ e_1(xy) = e_1(x) e_1(y) $, $
e_2(xy) = e_2(x) e_1(y) $. Thus, $ F_2 = \spn \{ e_1 \otimes e_1, \, e_2
\otimes e_1 \} $. A similar argument applies to $ F_3 $.

Case \ref{3.10} is similar to \ref{3.9}.
\end{proof}

\section[On graded algebras and their morphisms]
 {\raggedright On graded algebras and their morphisms}
\label{sec:5}
Given a graded algebra $ \A $, we consider the multiplication maps
\[
M_{s,t} : A_s \otimes \A_t \to \A_{s+t} \, , \quad M_{s,t} (x \otimes y) = xy
\]
for $ x \in \A_s, \, y \in \A_t $, and similarly
\[
M_{r,s,t} : A_r \otimes A_s \otimes \A_t \to \A_{r+s+t} \, , \quad M_{r,s,t} (x
\otimes y \otimes z) = xyz
\]
for $ x \in \A_r, \, y \in \A_s, \, z \in \A_t $. The kernel of $ M_{r,s,t} $
contains the sum of two linear subspaces,
\begin{equation}\label{5.05}
\Ker M_{r,s,t} \supset (\Ker M_{r,s}) \otimes \A_t + \A_r \otimes \Ker M_{s,t}
\end{equation}
since $ M_{r,s,t} = M_{r+s,t} ( M_{r,s} \otimes \One_t ) = M_{r,s+t} ( \One_r
\otimes M_{s,t} ) $. Sometimes we impose on $ \A $ the condition
\begin{equation}\label{5.2}
\Ker M_{r,s,t} = (\Ker M_{r,s}) \otimes \A_t + \A_r \otimes \Ker M_{s,t}
\quad \text{for all } r,s,t \in \{1,2,\dots\} \, .
\end{equation}
For the image of $ M_{s,t} $, the condition
\begin{equation}\label{5.3}
\Im M_{s,t} = \A_{s+t} \quad \text{for all } s,t \in \{1,2,\dots\}
\end{equation}
was mentioned in Sect.~1, see \eqref{cond_for_A}. We need also
\begin{gather*}
M^{(n)} : \A_1^{\otimes n} = \A_1 \otimes \dots \otimes \A_1 \to \A_n \, , \\
M^{(n)} ( x_1 \otimes \dots \otimes x_n ) = x_1 \dots x_n \quad \text{for }
 x_1,\dots,x_n \in \A_1 \, .
\end{gather*}
We consider algebras over the field $ \C $ of complex numbers only, and assume
that $ \A_0 = \C $ for each graded algebra $ \A $ under consideration.

\begin{proposition}\label{5.4}
Let graded algebras $ \A, \B $ be given together with linear maps
\[
\theta_1 : \A_1 \to \B_1 \, , \quad \theta_2 : \A_2 \to \B_2
\]
such that
\[
\theta_2 (xy) = \theta_1 (x) \theta_1 (y) \quad \text{for all } x,y \in \A_1
\, .
\]
Assume that $ \A $ satisfies \eqref{5.2} and \eqref{5.3}. Then there exist
unique linear maps $ \theta_t : \A_t \to \B_t $ for $ t = 3.4.\dots $ such
that $ \theta = (\theta_1,\theta_2,\theta_3,\theta_4,\dots) $ is a morphism of
graded algebras, that is,
\[
\theta_{s+t} (xy) = \theta_s (x) \theta_t (y) \quad \text{for all } x \in
\A_s, \, y \in \A_t, \, s,t \in \{1,2,3,4,\dots\} \, .
\]
\end{proposition}

The proof is given after several lemmas.

\begin{lemma}\label{5.5}
If $ \A $ satisfies \eqref{5.3} then $ \Im M^{(n)} = \A_n $ for all $ n =
2,3,\dots $
\end{lemma}

\begin{proof}
Induction in $ n $. First, $ M^{(2)} = M_{1,1} $, thus $ \Im M^{(2)} = \A_2 $
by \eqref{5.3}. Second, assume that $ \Im M^{(n)} = \A_n $. We have $
M^{(n+1)} = M_{n,1} ( M^{(n)} \otimes \One_1 ) $ and $ \Im ( M^{(n)} \otimes
\One_1 ) = \A_n \otimes \A_1 $, thus $ \Im M^{(n+1)} = \Im M_{n,1} = \A_{n+1}
$ by \eqref{5.3}.
\end{proof}

It is easy to see that
\begin{equation}\label{5.6}
\Im(\al) = F \quad \text{implies} \quad \al^{-1} (F_1+F_2) = \al^{-1} (F_1) +
\al^{-1} (F_2)
\end{equation}
whenever $ E,F $ are linear spaces, $ \al : E \to F $ a linear map, and $ F_1,
F_2 \subset F $ linear subspaces.

\begin{lemma}\label{5.7}
If $ \A $ satisfies \eqref{5.2} then $ \Ker M^{(n+1)} = ( \Ker M^{(n)} )
\otimes \A_1 + \A_1 \otimes \Ker M^{(n)} $ for all $ n = 2,3,\dots $
\end{lemma}

\begin{proof}
$ M^{(n+1)} = M_{1,n-1,1} ( \One_1 \otimes M^{(n-1)} \otimes \One_1 ) $, and $
\Im ( \One_1 \otimes M^{(n-1)} \otimes \One_1 ) = \A_1 \otimes \A_{n-1}
\otimes \A_1 $ by \ref{5.5}; using \eqref{5.2} and \eqref{5.6} we get
\begin{multline*}
\Ker M^{(n+1)} = ( \One_1 \otimes M^{(n-1)} \otimes \One_1 )^{-1} ( \Ker
 M_{1,n-1,1} ) = \\
= ( \One_1 \otimes M^{(n-1)} \otimes \One_1 )^{-1} \( (\Ker
 M_{1,n-1}) \otimes \A_1 + \A_1 \otimes \Ker M_{n-1,1} \) = \\
= ( \One_1 \otimes M^{(n-1)} \otimes \One_1 )^{-1} \( (\Ker M_{1,n-1}) \otimes
 \A_1 \) + \\
+ ( \One_1 \otimes M^{(n-1)} \otimes \One_1 )^{-1} \( \A_1 \otimes
 \Ker M_{n-1,1} \) = \\
= \underbrace{ ( \One_1 \otimes M^{(n-1)} )^{-1} ( \Ker M_{1,n-1} ) }_{=\Ker
  M^{(n)}} \otimes \A_1 + \A_1 \otimes \underbrace{ ( M^{(n-1)} \otimes \One_1
  )^{-1} ( \Ker M_{n-1,1} ) }_{=\Ker M^{(n)}} \, .
\end{multline*}
\end{proof}

\begin{proof}[Proof of Prop.~\ref{5.4}]
Uniqueness of $ \theta $ follows from \ref{5.5}, since $ \theta_n (x_1\dots
x_n) $ must be equal to $ \theta_1(x_1) \dots \theta_1(x_n) $ for all $ x_1,
\dots, x_n \in \A_1 $.

\begin{sloppypar}
Existence. It is sufficient to construct $ \theta_n $ such that $ \theta_n
(x_1\dots x_n) = \theta_1(x_1) \dots \theta_1(x_n) $ for all $ x_1,\dots, x_n
\in \A_1 $, since then $ \theta_{s+t} ( x_1 \dots x_s y_1 \dots y_t ) =
\theta_s ( x_1 \dots x_s ) \theta_t ( y_1 \dots y_t ) $ for all $
x_1,\dots,x_s, y_1,\dots,y_t \in \A_1 $ and therefore, using \ref{5.5}, $
\theta_{s+t} (xy) = \theta_s (x) \theta_t (y) $ for all $ x \in \A_s $, $ y
\in \A_t $. Thus, we need $ \theta_n $ that makes the diagram
\[
\xymatrix{
 \A_1^{\otimes n} \ar[r]^{M_\A^{(n)}} \ar[d]_{\theta_1^{\otimes n}} & \A_n
 \ar@{.>}[d]^{\theta_n}
\\
 \B_1^{\otimes n} \ar[r]^{M_\B^{(n)}} & \B_n
}
\]
commutative.
\end{sloppypar}

\begin{sloppypar}
By \ref{5.5}, $ \Im M_\A^{(n)} = \A_n $. It is sufficient to prove that
\[
\Ker M_\A^{(n)} \subset \Ker ( M_\B^{(n)} \theta_1^{\otimes n} ) \, .
\]
We do it by induction in $ n $. For $ n=2 $ it is ensured by the existence of
$ \theta_2 $. Let $ n > 2 $. By \ref{5.7}, $ \Ker M_\A^{(n)} = ( \Ker
M_\A^{(n-1)} ) \otimes \A_1 + \A_1 \otimes \Ker M_\A^{(n-1)} $ (since $ \A $
satisfies \eqref{5.2}). On the other hand, $ ( \Ker M_\B^{(n-1)} ) \otimes
\B_1 \subset \Ker M_\B^{(n)} $ and $ \B_1 \otimes \Ker M_\B^{(n-1)}  \subset
\Ker M_\B^{(n)} $ by \eqref{5.05}. Finally, $ \Ker M_\A^{(n-1)} \subset \Ker ( M_\B^{(n-1)}
\theta_1^{\otimes(n-1)} ) $ implies $ ( \Ker M_\A^{(n-1)} ) \otimes \A_1
\subset (\theta_1^{\otimes n})^{-1} \( ( \Ker M_\B^{(n-1)} ) \otimes \B_1 \) $
and $ \A_1 \otimes \Ker M_\A^{(n-1)} \subset (\theta_1^{\otimes n})^{-1} \(
\B_1 \otimes \Ker M_\B^{(n-1)} \) $.
\end{sloppypar}
\end{proof}

Let $ \A $ be a graded algebra and $ f $ its automorphism; that is, $ f_t :
\A_t \to \A_t $ is an invertible linear map, and $ f_{s+t}(xy) = f_s(x) f_t(y)
$ for all $ x \in \A_s $, $ y \in \A_t $ and $ s,t \in \{1,2,\dots\} $. Then
the binary operation
\[
\A_s \times \A_t \ni (x,y) \mapsto x f_t^s(y) \in \A_{s+t}
\]
(where $ f_t^s $ is $ f_t $ iterated $ s $ times) is associative:
\[
\( x f_s^r(y) \) f_t^{r+s}(z) = x f_{s+t}^r \( y f_t^s(z) \)
\]
for all $ x \in \A_r $, $ y \in \A_s $ and $ z \in \A_t $. Thus we get a new
graded algebra $ \A^{(f)} $.

\begin{lemma}\label{5.8}
If $ \A $ satisfies \eqref{5.3} then also $ \A^{(f)} $ satisfies \eqref{5.3}.
\end{lemma}

\begin{proof}
$ M_{s,t}^{(f)} = M_{s,t} ( \One_s \otimes f_t^s ) $ and $ \Im ( \One_s
\otimes f_t^s ) = \A_s \otimes \A_t $, thus $ \Im M_{s,t}^{(f)} = \Im M_{s,t}
= \A_{s+t} $.
\end{proof}

\begin{lemma}\label{5.9}
If $ \A $ satisfies \eqref{5.2} then also $ \A^{(f)} $ satisfies \eqref{5.2}.
\end{lemma}

\begin{proof}
Denoting for convenience $ \Ker M_{r,s} $ by $ K_{r,s} $ and $ f_s^{-1} $ by $
g_s $ (and similarly $ K_{s,t} $, $ g_t $) we have
\begin{align*}
& \Ker M_{r,s}^{(f)} = ( \One_r \otimes g_s^r ) K_{r,s} \, , & &\text{since }
 M_{r,s}^{(f)} = M_{r,s} ( \One \otimes f_s^r ) \, ; \\
& \Ker M_{s,t}^{(f)} = ( \One_s \otimes g_t^s ) K_{s,t} \, , &
 &\text{similarly;} \\
& \Ker M_{r,s,t}^{(f)} = ( \One_r \otimes g_s^r \otimes g_t^{r+s} ) \Ker
 M_{r,s,t} \, , & &\text{similarly;} \\
& \Ker M_{r,s,t} = K_{r,s} \otimes \A_t + \A_r \otimes K_{s,t} & &\text{by
 \eqref{5.2} for $ \A $;} \\
& ( g_s \otimes g_t ) K_{s,t} = K_{s,t} \, , & &\text{since } M_{s,t} ( f_s
\otimes f_t ) = f_{s+t} M_{s,t} \, .
\end{align*}
Thus,
\begin{multline*}
\Ker M_{r,s,t}^{(f)} = ( \One_r \otimes g_s^r \otimes g_t^{r+s} ) ( K_{r,s} \otimes \A_t + \A_r
 \otimes K_{s,t} ) = \\
= \( ( \One_r \otimes g_s^r ) K_{r,s} \) \otimes \A_t + \A_r \otimes ( \One_s \otimes g_t^s )
 \underbrace{ ( g_s^r \otimes g_t^r ) K_{s,t} }_{=K_{s,t}} = \\
= ( \Ker M_{r,s}^{(f)} ) \otimes \A_t + \A_r \otimes \Ker M_{s,t}^{(f)} \, .
\end{multline*}
\end{proof}

\section[Proof of the theorem]
 {\raggedright Proof of the theorem}
\label{sec:6}
Each two-dimensional algebra $ D $ given by Proposition \ref{3.14} is one of the
following four algebras: $ D_1 = D_2 = D_3 = D_4 = \C^2 $ as a linear space,
with the multiplication
\[
(a_1,a_2) (b_1,b_2) = \begin{cases}
 (a_1 b_1, a_2 b_2) &\text{for } D_1,\\
 (a_1 b_1, a_2 b_1 + a_1 b_2) &\text{for } D_2,\\
 (a_1 b_1, a_2 b_1) &\text{for } D_3,\\
 (a_1 b_1, a_1 b_2) &\text{for } D_4.
\end{cases}
\]
Algebras $ D_1 $ and $ D_2 $ contain units; $ D_3 $ and $ D_4 $ do not.

\begin{remark}
Similarly to the well-known classification of low-dimensional algebras with
unit \cite{St} one can classify all two-dimensional algebras (associative, not
necessarily with unit). They are four algebras $ D_1, D_2, D_3, D_4 $
satisfying \eqref{cond_for_B} and three algebras $ D_5, D_6, D_7 $ violating
\eqref{cond_for_B}:
\[
(a_1,a_2) (b_1,b_2) = \begin{cases}
 (a_1 b_1, 0) &\text{for } D_5,\\
 (0, a_1 b_1) &\text{for } D_6,\\
 (0, 0) &\text{for } D_7.
\end{cases}
\]
However, this classification is not used in the present work.
\end{remark}

\begin{lemma}\label{6.1}
The condition
\begin{equation}\label{6.2}
\Ker \mu_3 = ( \Ker \mu_2 ) \otimes D + D \otimes \Ker \mu_2
\end{equation}
is satisfied by each one of the algebras $ D_1, D_2, D_3, D_4 $; here $ \mu_2
: D \otimes D \to D $ and $ \mu_3 : D \otimes D \otimes D \to D $ are the
multiplication maps.
\end{lemma}

\begin{proof}
First, the condition is satisfied by every algebra $ D $ with unit, since
\begin{gather*}
x \otimes y - 1 \otimes (xy) \in \Ker \mu_2 \, ; \\
x \otimes y \otimes z - 1 \otimes (xy) \otimes z \in ( \Ker \mu_2 ) \otimes D
 \, ; \\
(xy) \otimes z - (xyz) \otimes 1 \in \Ker \mu_2 \, ; \\
1 \otimes (xy) \otimes z - 1 \otimes (xyz) \otimes 1 \in D \otimes \Ker \mu_2 \, ; \\
x \otimes y \otimes z - 1 \otimes (xyz) \otimes 1 \in ( \Ker \mu_2 ) \otimes D
 + D \otimes \Ker \mu_2
\end{gather*}
for all $ x,y,z \in D $. That is,
\[
u - 1 \otimes ( \mu_3 u ) \otimes 1 \in ( \Ker \mu_2 ) \otimes D + D \otimes
\Ker \mu_2
\]
for all $ u \in D \otimes D \otimes D $. Thus, $ u \in \Ker \mu_3 $ implies $
u \in ( \Ker \mu_2 ) \otimes D + D \otimes \Ker \mu_2 $.

It remains to check the condition for $ D_3 $ ($ D_4 $ is similar). In terms
of the basis vectors $ e_1=(1,0) $ and $ e_2=(0,1) $ we have $ e_1 e_1 = e_1
$, $ e_2 e_1 = e_2 $, $ e_1 e_2 = 0 $, $ e_2 e_2 = 0 $, thus $ \Ker \mu_2 = D
\otimes e_2 $. Denote $ E = \spn \{ e_1 \otimes e_2, e_2 \otimes e_1, e_2
\otimes e_2 \} $ and $ F = \spn \{ e_1 \otimes e_1 \} $, then $ D \otimes D =
E \oplus F $ and $ ( \Ker \mu_2 ) \otimes D + D \otimes \Ker \mu_2 = D \otimes
e_2 \otimes D + D \otimes D \otimes e_2 = D \otimes E $. It remains to check
that $ \Ker \mu_3 = D \otimes E $. On one hand, $ \Ker \mu_3 \supset D \otimes
E $. On the other hand, $ \Ker \mu_3 \cap ( D \otimes F ) = \{0\} $, since $ x
e_1 e_1 = x $ for all $ x \in D $.
\end{proof}

\begin{proof}[Proof of Theorem \ref{main_theorem}]
Let $ \A $ be a graded algebra satisfying \eqref{cond_for_A}. The
multiplication maps $ M_{1,1} : \A_1 \otimes \A_1 \to \A_2 $ and $ M_{1,1,1} :
\A_1 \otimes \A_1 \otimes \A_1 \to \A_3 $ are surjective (by
\eqref{cond_for_A} and \ref{5.5}); the dual maps $ M'_{1,1} : \A'_2 \to
\A'_1 \otimes \A'_1 $ and $ M'_{1,1,1} : \A'_3 \to \A'_1 \otimes \A'_1
\otimes \A'_1 $ are injective; thus, the spaces
\begin{align*}
E_1 &= \A'_1 \, , \\
E_2 &= \Im M'_{1,1} \subset E_1 \otimes E_1 \, , \\
E_3 &= \Im M'_{1,1,1} \subset E_1 \otimes E_1 \otimes E_1
\end{align*}
are two-dimensional. Their annihilators $ E_2^\perp = \Ker M_{1,1} $, $
E_3^\perp = \Ker M_{1,1,1} $ satisfy $ E_3^\perp \supset E_2^\perp \otimes \A_1
+ \A_1 \otimes E_2^\perp $ by \eqref{5.05}; therefore $ E_3 \subset ( E_2
\otimes E_1 ) \cap ( E_1 \otimes E_2 ) $, which shows that the triple $
(E_1,E_2,E_3) $ satisfies \eqref{3.25}.

\begin{sloppypar}
By Proposition \ref{3.14} the triple $ (E_1,E_2,E_3) $ is isomorphic to a triple $
(F_1,F_2,F_3) $ resulting from a two-dimensional algebra $ D $ satisfying
\eqref{cond_for_B} and its automorphism $ \eta $ by a construction that is in
fact a fragment of the construction \eqref{the_construction}. That is, we
construct a graded algebra $ \B $ satisfying \eqref{cond_for_A} out of $ D $
and $ \eta $ by \eqref{the_construction}. Then we construct
\begin{align*}
F_1 &= \B'_1 \, , \\
F_2 &= \Im ( M_{1,1}^\B )' \subset F_1 \otimes F_1 \, , \\
F_3 &= \Im ( M_{1,1,1}^\B )' \subset F_1 \otimes F_1 \otimes F_1
\end{align*}
($ M_{1,1}^\B : \B_1 \otimes \B_1 \to \B_2 $ and $ M_{1,1,1}^\B : \B_1 \otimes
\B_1 \otimes \B_1 \to \B_3 $ being the multiplication maps) and conclude that
the triple $ (F_1,F_2,F_3) $ is isomorphic (as defined in Sect.~\ref{sec:4})
to $ (E_1,E_2,E_3) $.
\end{sloppypar}

It remains to prove that the graded algebras $ \A $ and $ \B $ are
isomorphic.

We have an isomorphism between $ (F_1,F_2,F_3) $ and $ (E_1,E_2,E_3) $, --- an
invertible linear map $ \al : E_1 \to F_1 $ such that $ ( \al \otimes \al )
E_2 = F_2 $.\footnote{%
 We also have $ ( \al \otimes \al \otimes \al ) E_3 = F_3 $ but we do not use
 it, and in fact, we'll get it again (implicitly) via Prop.~\ref{5.4}.}
However, $ E_1 = \A'_1 $, $ F_1 = \B'_1 $, $ E_2 = \Im (M_{1,1}^\A)' = ( \Ker
M_{1,1}^\A )^\perp $ and $ F_2 = \Im (M_{1,1}^\B)' = ( \Ker M_{1,1}^\B )^\perp
$. We introduce the dual map $ \theta_1 = \al' $; it satisfies
\begin{gather*}
\theta_1 : \B_1 \to \A_1 \, , \\
( \theta_1 \otimes \theta_1 ) \Ker M_{1,1}^\B = \Ker M_{1,1}^\A \, .
\end{gather*}
Taking into account that $ \Im M_{1,1}^\B = \B_2 $ (since $ \B $ satisfies
\eqref{cond_for_A}) we get $ \theta_2 : \B_2 \to \A_2 $ such that $ \theta_2
(xy) = \theta_1(x) \theta_1(y) $ for all $ x,y \in \B_1 $.
\[
\xymatrix{
 \B_1 \otimes \B_1 \ar[r]^{M_{1,1}^\B} \ar[d]_{\theta_1\otimes\theta_1} & \B_2
 \ar@{.>}[d]^{\theta_2}
\\
 \A_1 \otimes \A_1 \ar[r]^{M_{1,1}^\A} & \A_2
}
\]

The algebra $ \B $ satisfies \eqref{cond_for_A}, therefore,
\eqref{5.3}. Condition \eqref{5.2} is satisfied by $ \B $ by \ref{5.9} and
\ref{6.1}. Thus, Proposition \ref{5.4} extends $ \theta_1 $ (and $ \theta_2 $)
to a morphism $ \theta : \B \to \A $. It remains to prove that $ \theta $ is
an isomorphism.

Also $ \A $ satisfies \eqref{cond_for_A}, therefore, \eqref{5.3}. By Lemma
\ref{5.5}, $ \Im M_\A^{(n)} = \A_n $.
Using commutativity of the diagram
\[
\xymatrix{
 \B_1^{\otimes n} \ar[r]^{M_\B^{(n)}} \ar[d]_{\theta_1^{\otimes n}} & \B_n
 \ar[d]^{\theta_n}
\\
 \A_1^{\otimes n} \ar[r]^{M_\A^{(n)}} & \A_n
}
\]
and surjectivity of $ \theta_1 $ we see that $ \theta_n $ maps the
two-dimensional space $ \B_n $ onto the two-dimensional space $ \A_n
$. Therefore $ \theta_n $ is invertible, and so, the morphism $ \theta $ is an
isomorphism.
\end{proof}

\section[Classification]
 {\raggedright Classification}
\label{sec:7}
Combining \eqref{the_construction}, the cases constituting the proof of Proposition
\ref{3.14}, and Lemma \ref{3.13} we get the following classification results.

\begin{center}\textsc{graded algebras}\end{center}

Every graded algebra satisfying \eqref{cond_for_A} is isomorphic to one and
only one of the graded algebras $ \B_1 $, $ \B_2 $, $ \B_3(\la) $, $ \B_4 $, $
\B_5 $ defined below ($ \la $ runs over $ \C \setminus \{0\} $). In all cases
$ A_0 = \C $ and $ \A_t = \C^2 $ for $ t = 1,2,\dots $ For every $ s,t \in
\{1,2,\dots\} $ the product of $ (a_1,a_2) \in \A_s $ and $ (b_1,b_2) \in \A_t
$ is $ (c_1,c_2) \in \A_{s+t} $ where
\begin{align*}
& c_1 = a_1 b_1 \, , \quad c_2 = a_2 b_2 & & \text{for } \B_1 \, , \\
& \begin{cases}
 c_1 = a_1 b_1 \, , \; c_2 = a_2 b_2 &\text{if $ s $ is even},\\
 c_1 = a_1 b_2 \, , \; c_2 = a_2 b_1 &\text{if $ s $ is odd}
\end{cases} & & \text{for } \B_2 \, , \\
& c_1 = a_1 b_1 \, , \quad c_2 = a_2 b_1 + \la a_1 b_2 & & \text{for } \B_3(\la) \, , \\
& c_1 = a_1 b_1 \, , \quad c_2 = a_2 b_1 & & \text{for } \B_4 \, , \\
& c_1 = a_1 b_1 \, , \quad c_2 = a_1 b_2 & & \text{for } \B_5 \, .
\end{align*}

\begin{center}\textsc{subproduct systems}\end{center}

Every subproduct system (as defined in \ref{subproduct_system}) is isomorphic
to one and only one of the subproduct systems $ \Ec_1 $, $ \Ec_2 $, $ \Ec_3(\la)
$, $ \Ec_4 $, $ \Ec_5 $ defined below ($ \la $ runs over $ \C \setminus \{0\}
$). In all cases $ E_t = \C^2 $ for $ t = 1,2,\dots $, and $ e_1 = (1,0) $, $
e_2 = (0,1) $ are the basis vectors. For every $ s,t \in \{1,2,\dots\} $
\begin{align*}
& \be_{s,t} (e_1) = e_1 \otimes e_1 \, , \quad \be_{s,t} (e_2) = e_2 \otimes e_2
 & & \text{for } \Ec_1 \, , \\
& \be_{s,t} (e_1) = \begin{cases}
 e_1 \otimes e_1 &\text{if $ s $ is even},\\
 e_1 \otimes e_2 &\text{if $ s $ is odd},
\end{cases} \quad
\be_{s,t} (e_2) = \begin{cases}
 e_2 \otimes e_2 &\text{if $ s $ is even},\\
 e_2 \otimes e_1 &\text{if $ s $ is odd}
\end{cases}
 & & \text{for } \Ec_2 \, , \\
& \be_{s,t} (e_1) = e_1 \otimes e_1 \, , \quad \be_{s,t} (e_2) = e_2 \otimes e_1
 + \la^s e_1 \otimes e_2 & & \text{for } \Ec_3(\la) \, , \\
& \be_{s,t} (e_1) = e_1 \otimes e_1 \, , \quad \be_{s,t} (e_2) = e_2 \otimes e_1
 & & \text{for } \Ec_4 \, , \\
& \be_{s,t} (e_1) = e_1 \otimes e_1 \, , \quad \be_{s,t} (e_2) = e_1 \otimes e_2
 & & \text{for } \Ec_5 \, .
\end{align*}

\appendix

\section[Appendix: calculating determinants]
 {\raggedright Appendix: calculating determinants}
\label{sec:a}
A proof of the equality
\[
D_8 + D_4 = 0
\]
is sketched here. Recall that $ D_8 $ is defined by \eqref{det8}, and $ D_4 $
is defined by \eqref{2.5}--\eqref{2.7}. The equality is stated for all $ a $,
$ b $, $ c $, $ d $, $ e $, $ f $, $  g $, $ h $, $ A $, $ B $, $ C $, $ D $,
$ E $, $ F $, $ G $, $ H $.

We use the Laplace expansion\footnote{%
 See e.g. \cite[Sect.~1.6, page 50]{Ku}.}
by the first four rows:
\[
D_8 = \sum_{i,j,k,l} (-1)^{i+j+k+l} A_{i,j,k,l} B_{i,j,k,l}
\]
where the sum is taken over $ 1 \le i < j < k < l \le 8 $; $ A_{i,j,k,l} $ is
the minor in rows $ 1,2,3,4 $ and columns $ i,j,k,l $; and $ B_{i,j,k,l} $ is
the complementary minor (in rows $ 5,6,7,8 $ and the remaining columns). The
sum contains $ \binom 8 4 $ terms, but most of them vanish. Namely, $
A_{i,j,k,l} \ne 0 $ only if $ | \{ i,j,k,l \} \cap \{ 1,2,3,4 \} | = 2 $ (that
is, exactly two of $ i,j,k,l $ do not exceed $ 4 $), since otherwise the minor
contains two linearly dependent rows (each with a single non-zero
element). Similarly, $ B_{i,j,k,l} \ne 0 $ only if $ | \{ i,j,k,l \} \cap \{
1,2,5,6 \} | = 2 $. Thus, $ 18 $ terms survive:
\begin{align*}
& i=1, \, j=2, \, k=7, \, l=8; \quad & & \text{1 term} \\
& i=3, \, j=4, \, k=5, \, l=6; \quad & & \text{1 term} \\
& 1 \le i \le 2, \, 3 \le j \le 4, \, 5 \le k \le 6, \, 7 \le l \le 8. \quad & & \text{16 terms} \\
\end{align*}
We have
\[
A_{1,2,7,8} = A_{3,4,5,6} =
 \begin{vmatrix} a & b \\ e & f \end{vmatrix}
 \begin{vmatrix} c & d \\ g & h \end{vmatrix}
 \, , \quad
B_{1,2,7,8} = B_{3,4,5,6} =
 \begin{vmatrix} A & B \\ E & F \end{vmatrix}
 \begin{vmatrix} C & D \\ G & H \end{vmatrix}
 \, ,
\]
thus the first $ 2 $ terms contribute
\[
S_2 = A_{1,2,7,8} B_{1,2,7,8} + A_{3,4,5,6} B_{3,4,5,6} = 2
 \begin{vmatrix} a & b \\ e & f \end{vmatrix}
 \begin{vmatrix} c & d \\ g & h \end{vmatrix}
 \begin{vmatrix} A & B \\ E & F \end{vmatrix}
 \begin{vmatrix} C & D \\ G & H \end{vmatrix}
 \, .
\]
Using $ p,q,r $ of \eqref{2.5} and in addition
\[
q_1 = \begin{vmatrix} a & d \\ e & h \end{vmatrix} \, ,
\quad q_2 = \begin{vmatrix} b & c \\ f & g \end{vmatrix}
\]
we have for $ 1 \le i \le 2 $, $ 3 \le j \le 4 $, $ 5 \le k \le 6 $, $ 7 \le l
\le 8 $
\begin{gather*}
A_{i,j,k,l} = \al_{i,j} \al_{k-4,l-4} \quad \text{where} \\
\al_{1,3} = p \, , \quad \al_{1,4} = q_1 \, , \quad \al_{2,3} = q_2 \, , \quad
 \al_{2,4} = r \, .
\end{gather*}
Similarly,
\begin{gather*}
B_{i,j,k,l} = - \be_{i,k} \be_{j-2,l-2} \quad \text{where} \\
\be_{1,5} = R \, , \quad \be_{1,6} = Q_2 \, , \quad \be_{2,5} = Q_1 \, , \quad
 \be_{2,6} = P \, .
\end{gather*}
Thus the last $ 16 $ terms contribute
\begin{alignat*}{4}
S_{16} = & - ppRR && + pq_1 RQ_2 && + pq_2 Q_2R && - pr Q_2Q_2 \\
         & + q_1p RQ_1 && - q_1q_1 RP && - q_1q_2 Q_2Q_1 && + q_1r Q_2P \\
         & + q_2p Q_1R && - q_2q_1 Q_1Q_2 && - q_2q_2 PR && + q_2r PQ_2 \\
         & - rp Q_1Q_1 && + rq_1 Q_1P && + rq_2 PQ_1 && - rr PP = S_+ - S_- \, ,
\end{alignat*}
where $ S_+ $ is the sum of the $ 8 $ ``positive'' terms (with the plus sign)
and $ (-S_-) $ is the sum of the $ 8 $ ``negative'' terms. Taking into account
that $ q_1 + q_2 = q $ and $ Q_1 + Q_2 = Q $ we combine the $ 8 $ ``positive''
terms into
\[
S_+ = pqQR + qrPQ \, .
\]
Other terms give
\[
S_- = p^2 R^2 + r^2 P^2 + pr (Q_1^2+Q_2^2) + (q_1^2+q_2^2) PR + 2 q_1q_2
Q_1Q_2 \, .
\]
We expand $ D_4 $ by the first and third columns,
\begin{multline*}
D_4 = \\
      - \begin{vmatrix} p & r \\ 0 & q \end{vmatrix}
        \begin{vmatrix} Q & 0 \\ P & R \end{vmatrix} +
        \begin{vmatrix} p & r \\ P & R \end{vmatrix}
        \begin{vmatrix} p & r \\ P & R \end{vmatrix} -
        \begin{vmatrix} p & r \\ 0 & Q \end{vmatrix}
        \begin{vmatrix} p & r \\ Q & 0 \end{vmatrix} -
        \begin{vmatrix} 0 & q \\ P & R \end{vmatrix}
        \begin{vmatrix} q & 0 \\ P & R \end{vmatrix} -
        \begin{vmatrix} P & R \\ 0 & Q \end{vmatrix}
        \begin{vmatrix} q & 0 \\ p & r \end{vmatrix} \\
=       \begin{vmatrix} p & r \\ P & R \end{vmatrix}^2 -
        \begin{vmatrix} p & q \\ P & Q \end{vmatrix}
        \begin{vmatrix} q & r \\ Q & R \end{vmatrix} \, .
\end{multline*}
Thus,
\begin{multline*}
D_4 + S_{16} = ( p^2 R^2 - 2pr PR + r^2 P^2 - pq QR + pr Q^2 + q^2 PR - qr PQ
 ) + \\
+ ( pq QR + qr PQ ) - \\
- \( p^2 R^2 + r^2 P^2 + pr ( Q_1^2 + Q_2^2 ) + ( q_1^2 + q_2^2 ) PR + 2 q_1
 q_2 Q_1 Q_2 ) = \\
= pr ( Q^2 - Q_1^2 - Q_2^2 ) + ( q^2 - q_1^2 - q_2^2 ) PR - 2 pr PR - 2 q_1
 q_2 Q_1 Q_2 = \\
= 2 pr Q_1 Q_2 + 2 q_1 q_2 PR - 2 pr PR - 2 q_1 q_2 Q_1 Q_2 = \\
= -2 ( pr - q_1 q_2 ) ( PR - Q_1 Q_2 ) \, .
\end{multline*}
However,
\[
pr - q_1 q_2 = \begin{vmatrix} a & c \\ e & g \end{vmatrix}
               \begin{vmatrix} b & d \\ f & h \end{vmatrix} -
               \begin{vmatrix} a & d \\ e & h \end{vmatrix}
               \begin{vmatrix} b & c \\ f & g \end{vmatrix} =
               \begin{vmatrix} c & d \\ g & h \end{vmatrix}
               \begin{vmatrix} a & b \\ e & f \end{vmatrix} \, ;
\]
the same holds for $ PR - Q_1 Q_2 $. Finally,
\begin{multline*}
D_8 + D_4 = S_2 + S_{16} + D_4 = \\
= 2 \begin{vmatrix} a & b \\ e & f \end{vmatrix}
  \begin{vmatrix} c & d \\ g & h \end{vmatrix}
  \begin{vmatrix} A & B \\ E & F \end{vmatrix}
  \begin{vmatrix} C & D \\ G & H \end{vmatrix} -
2 \begin{vmatrix} c & d \\ g & h \end{vmatrix}
  \begin{vmatrix} a & b \\ e & f \end{vmatrix} \cdot
  \begin{vmatrix} C & D \\ G & H \end{vmatrix}
  \begin{vmatrix} A & B \\ E & F \end{vmatrix} = 0 \, .
\end{multline*}

\bigskip
\filbreak
{
\small
\begin{sc}
\parindent=0pt\baselineskip=12pt
\parbox{4in}{
Boris Tsirelson\\
School of Mathematics\\
Tel Aviv University\\
Tel Aviv 69978, Israel
\smallskip
\par\quad\href{mailto:tsirel@post.tau.ac.il}{\tt
 mailto:tsirel@post.tau.ac.il}
\par\quad\href{http://www.tau.ac.il/~tsirel/}{\tt
 http://www.tau.ac.il/\textasciitilde tsirel/}
}

\end{sc}
}
\filbreak

\end{document}